\documentclass[12pt]{amsart}
\usepackage{verbatim}
\usepackage{amssymb}

\newtheorem{theorem}{Theorem}[section]
\newtheorem{proposition}[theorem]{Proposition}

\newtheorem{lemma}[theorem]{Lemma}
\newtheorem{claim}[theorem]{Claim}

\newtheorem{definition}[theorem]{Definition}

\newtheorem{conjecture}[theorem]{Conjecture}

\theoremstyle{plain}
\numberwithin{equation}{section}

\theoremstyle{remark}
\newtheorem{remark}[theorem]{Remark}

\newcommand{\C}{{\mathbb C}}

\newcommand{\Q}{{\mathbb Q}}
\newcommand{\R}{{\mathbb R}}
\newcommand{\Z}{{\mathbb Z}}
\newcommand{\N}{{\mathbb N}}

\newcommand{\cV}{{\mathcal V}}

\newcommand{\Pl}{{\mathbb P}}

\DeclareMathOperator{\bN}{\mathbb{N}}
\DeclareMathOperator{\Res}{Res}

\DeclareMathOperator{\lcm}{lcm}

\newcommand{\bP}{{\mathbb P}}
\newcommand{\bZ}{{\mathbb Z}}

\newcommand{\bC}{{\mathbb C}}

\newcommand{\bQ}{{\mathbb Q}}
\newcommand{\bF}{{\mathbb F}}

\newcommand{\lra}{\longrightarrow}

\newcommand{\cO}{\mathcal{O}}

\newcommand{\cC}{\mathcal{C}}

\newcommand{\cS}{\mathcal{S}}

\newcommand{\Zp}{\bZ_p}
\newcommand{\Qp}{\bQ_p}

\newcommand{\Cp}{\bC_p}
\newcommand{\PCp}{\bP^1(\Cp)}

\newcommand{\Dbar}{\overline{D}}

\newcommand{\rmkd}[1]{}

\newcommand{\dsps}{\displaystyle}

\title{A gap principle for dynamics}

\author[Benedetto, Ghioca, Kurlberg, Tucker]
{Robert L.  Benedetto, Dragos Ghioca, P\"{a}r Kurlberg,
  and Thomas J. Tucker}

\date{October 6, 2008}


\address{Department of Mathematics and Computer Science \\
        Amherst College \\
        Amherst, MA 01002 \\
        USA}
\email{rlb@cs.amherst.edu}

\address{Department of Mathematics and Computer Science\\
University of Lethbridge\\
Lethbridge, AB T1K 3M4\\
Canada}
\email{dragos.ghioca@uleth.ca}

\address{
Department of Mathematics\\
KTH\\
SE-100 44 Stockholm\\
Sweden}
\email{kurlberg@math.kth.se}

\address{Department of Mathematics\\
University of Rochester\\
Rochester, NY 14627\\ 
USA}
\email{ttucker@math.rochester.edu}

\begin{document}


\begin{abstract}
  Let $f_1,\ldots,f_g\in {\mathbb C}(z)$ be rational functions, let
  $\Phi=(f_1,\ldots,f_g)$ denote their coordinatewise action on
  $({\mathbb P}^1)^g$, let $V\subset ({\mathbb P}^1)^g$ be a proper subvariety, and
  let $P=(x_1,\ldots,x_g)\in ({\mathbb P}^1)^g({\mathbb C})$ be a nonpreperiodic point
  for $\Phi$.  We show that if $V$ does not contain any periodic
  subvarieties of positive dimension, then the set of $n$ such that
  $\Phi^n(P) \in V({\mathbb C})$ must be very sparse.  In particular, for any
  $k$ and any sufficiently large $N$, the number of $n \leq N$ such that
  $\Phi^n(P) \in V({\mathbb C})$ is less than $\log^k N$, where $\log^k$
  denotes the $k$-th iterate of the $\log$ function.  This can be
  interpreted as an analog of the gap principle of Davenport-Roth and
  Mumford.
\end{abstract}

\maketitle

\section{Introduction}
\label{sec:introduction}

The Mordell-Lang conjecture proved by Faltings \cite{Faltings} and
Vojta \cite{V1} implies that if $V$ is a subvariety of a semiabelian
variety $G$ defined over $\C$ such that $V$ contains no translate of a
positive-dimensional algebraic subgroup of $G$, then $V(\C)$ contains
at most finitely many points of any given finitely generated subgroup
$\Gamma$ of $G(\C)$. A reformulation of this result says that if no
translate of $V$ contains a positive-dimensional subvariety $W$ which
is fixed by the multiplication-by-$n$-map (for any positive integer
$n\ge 2$), then $V(\C)\cap\Gamma$ is finite (see \cite[Lemma
3]{Abramovich}).

In \cite{new-log}, Ghioca and Tucker proposed a dynamical analogue of
the Mordell-Lang conjecture (see also \cite{Denis-dynamical} and
\cite{Bell}).

\begin{conjecture}
\label{dynamical Mordell-Lang}
Let $X$ be a quasiprojective variety defined over $\C$, let $V\subset
X$ be any subvariety, let $\Phi:X\lra X$ be any endomorphism, and let
$P\in X(\C)$. For any integer $m\geq 0$, denote by
  $\Phi^m$ the $m^{\text{th}}$ iterate $\Phi\circ\cdots\circ\Phi$.  Then
$\{n\geq 0 : \Phi^n(P)\in V(\C)\}$ is a union of at most finitely many
arithmetic progressions and at most finitely many other integers.
\end{conjecture}

A special case of the above conjecture is our
Conjecture~\ref{nonperiodic}. Before stating it, we need the following
definition.

\begin{definition}
\label{periodic varieties}
Let $X$ be a quasiprojective variety, let $\Phi:X\lra X$ be
an endomorphism, 
let $P$ be a point on $X$,
and let $V\subset X$ be a subvariety.
The {\em orbit} of $P$ under $\Phi$ is
$\cO_{\Phi}(P)=\{\Phi^n(P) : n\geq 0\}$.
We say $V$ is {\em periodic} under $\Phi$ if there is a
positive integer $N\geq 1$ such that $\Phi^N(V)\subseteq V$.
\end{definition}

We will often omit the phrase ``under $\Phi$'' if the meaning
is clear from context.
We say that
$P$ is \emph{preperiodic} if $\cO_{\Phi}(P)$ is finite.

\begin{conjecture}
\label{nonperiodic}
Let $X$ be a quasiprojective variety defined over $\C$, let $V\subset
X$ be a subvariety, let $\Phi$ be an endomorphism of $X$, and let $P\in X(\C)$.
If $V(\C)\cap\cO_{\Phi}(P)$ is an infinite set, then
$V$ contains a positive-dimensional subvariety that
is periodic under $\Phi$.
\end{conjecture}

Conjecture~\ref{dynamical Mordell-Lang} implies
Conjecture~\ref{nonperiodic}, as follows.
If $V(\C)\cap\cO_{\Phi}(P)$ is infinite, then
$V$ contains $\cO_{\Phi^N}(\Phi^{\ell}(P))$, which is
also infinite, for some positive integers $N$ and $\ell$,
by Conjecture~\ref{dynamical Mordell-Lang}.
Thus, $V$ also contains the Zariski
closure $W$ of $\cO_{\Phi^N}(\Phi^{\ell}(P))$; clearly
$\Phi^N(W)\subset W$, and hence $W$ is periodic.

In this paper we consider the case that $X=(\bP^1)^g$
and $\Phi$ is of the form
$\Phi(z_1,\ldots,z_g)=(f_1(z_1),\ldots,f_g(z_g))$,
and we prove a weak form of Conjecture~\ref{nonperiodic}:
either the conclusion of Conjecture~\ref{nonperiodic} holds, or
the set $\{n\geq 0: \Phi^n(P)\in V(\C)\}$ is \emph{very thin}.

\begin{theorem}
\label{thm:ml}
Let $f_1,\ldots,f_g\in \bC(z)$ be
rational functions, and let $\Phi=(f_1,\ldots,f_g)$
denote their coordinatewise action on
$(\bP^1)^g$.  Let $P=(x_1,\ldots,x_g)\in (\bP^1)^g(\bC)$ be a
nonpreperiodic point for $\Phi$, and let $V\subset (\bP^1)^g$
be a proper subvariety.
Then one of the following two statements holds:
\begin{itemize}
\item[(i)] $V$ contains a positive-dimensional periodic
subvariety that intersects $\cO_{\Phi}(P)$.
\item[(ii)] There are integers $T,N \geq 1$ and a constant $C > 1$
such that for all integers $n>m\geq 0$
and $\ell \in \{T+1,T+2,\ldots,T+N \}$
for which both $\Phi^{\ell + mN}(P), \Phi^{\ell + nN}(P) \in V$,
we have $n-m > C^m$.
\end{itemize}
\end{theorem}

\begin{remark}
At the expense of replacing $C$ with a smaller
constant (but still larger
than $1$), conclusion~(ii) of Theorem~\ref{thm:ml} may be rephrased as:
\begin{itemize}
\item[$(\text{ii}^{'})$] For any sufficiently large integers $n>m\geq 0$ such that
$n \equiv m \pmod{N}$ and $\Phi^m(P), \Phi^n(P)\in V$,
we have $n-m >  C^m$.
\end{itemize}
In fact, if $N>1$, by replacing $C$ with an even  smaller
constant (but still larger
than $1$), conclusion~(ii) implies
\begin{itemize}
\item[$(\text{ii}^{''})$] For {\em all} integers $n>m\geq 0$ such that
$n \equiv m \pmod{N}$ and $\Phi^m(P), \Phi^n(P)\in V$,
we have $n-m >  C^m$.
\end{itemize}
\end{remark}

If $V$ is a curve defined over a number field, then we can prove the following more precise result.
\begin{theorem}
\label{thm:controlling-c-n-for-curves}
Let $P$, $\Phi$, and $V$ be as in Theorem~\ref{thm:ml}.
Assume further that $V$ is an irreducible {\em curve} that is not periodic, and
that both $V$ and $P$ are defined over a number field $K$.  Then
for any
$\epsilon>0$, there are infinitely many primes $p$ and associated
constants $C=C(p)>p-\epsilon$ and $N=N(p) = O(p^{2[K:\Q]})$ with the
following property:
For any integers $n>m\geq 0$ and $\ell \in \{1,2,\ldots,N\}$,
if $m$ is sufficiently large
and if both $\Phi^{\ell + mN}(P),\Phi^{\ell + nN}(P) \in V$, then
$n -m > C^{m}$.
\end{theorem}

Theorem~\ref{thm:ml} says that
unless $V$ contains a positive-dimensional periodic subvariety,
the integers $n$ such that $\Phi^n(P) \in V$ grow {\em very rapidly}.
To describe this growth more explicitly we first recall Knuth's
``up-arrow'' notation.  Given $C > 1$, define $C \uparrow \uparrow m$
for integers $m\geq 1$ as follows: $C \uparrow \uparrow 1 := C$; and
for $m\geq 2$, set $ C \uparrow \uparrow m := C^{ C \uparrow \uparrow
(m-1)}$.  It follows from Theorem~\ref{thm:ml} that if $n_i$ is the
$i^{\rm th}$ integer
in a given congruence class mod $N$ for which
$\Phi^{n_i}(P) \in V$, then $n_i > C \uparrow \uparrow (i-T)$ for all
$i>T$, where $C$, $T$ and $N$ are the constants in
Theorem~\ref{thm:ml}.
The growth condition might also be formulated without restricting to
congruence classes: if $n_{i}$ is the $i^{th}$ integer such that $\Phi^{n_{i}}(P)
\in V$, then $n_{i} > C\uparrow \uparrow \lfloor(i-T)/N \rfloor$ for $i>T$.
In particular, $n_i$ grows much faster than $\exp^k(i)$ for any $k \ge 1$, where $\exp^k$ denotes the 
$k^{\rm th}$ iterate of the exponential function.

We may also rephrase Theorem~\ref{thm:ml} in terms of extremely
{\em  slow} growth of the counting function for the number of indices
$n$ such that $\Phi^{n}(P) \in V$.  To do so, we set the following
notation.
Given $Y,C > 1$, define $L_C(Y)$ to be the smallest integer $m$
such that $(C \uparrow \uparrow m) > Y$.
In particular, note that for any $k$, $L_{C}(Y)$ grows slower
than the $k$-fold iterated logarithm.
\begin{theorem}
\label{thm:main-counting}
Let $P$, $\Phi$, and $V$ be as in Theorem~\ref{thm:ml}.
Set $\cS=\{n\geq 0 : \Phi^n(P)\in V\}$.
Then either 
$V$ contains a positive-dimensional periodic subvariety,
or there are  constants $N,C>1$ such that
\begin{equation*}
|\{ n\in\cS : n\leq M\}|   \leq N \cdot L_C(M) + O_{V,\Phi,P}(1).
\end{equation*}

\end{theorem}

Denis \cite{Denis-dynamical} has treated the question of
the distribution of the set $\cS$
when $V$ does not contain a periodic subvariety.  He showed,
for any morphism $\Phi$ of varieties over a field of
characteristic 0, that $\cS$ cannot be ``very dense of order 2''
(see \cite[D\'efinition 2]{Denis-dynamical}).
Theorem~\ref{thm:main-counting} shows that $\cS$ satisfies a much
stronger nondensity condition in the case that the morphism is
a product of self-maps of the projective line.

When our points and maps are defined over a number field $K$, we may
phrase this discussion in terms of (logarithmic) Weil heights;
see \cite[Ch. 3]{silverman-arithmetic-dynamics-book} for background
on heights.
If $P$ is not preperiodic,
then the Weil height $h(\Phi^n(P))$ grows
{\em at least} as $\deg_{\min}(\Phi)^n$,
where $\deg_{\min}(\Phi):=\min_j\deg(f_j)$.
Thus, we obtain:

\begin{theorem}
Let $P$, $\Phi$, and $V$ be as in Theorem~\ref{thm:ml},
and let $n_i$ denote
the $i^{\rm th}$ integer $n$ such that $\Phi^{n}(P) \in V$.
Assume that $\deg_{\min}(\Phi)\ge 2$, and that
$V$ does not contain a periodic subvariety that intersects $\cO_{\Phi}(P)$.
Then there are constants $T,N\geq 1$ and $C>1$ such that
  $h(\Phi^{n_i}(P))$ grows faster than $C \uparrow \uparrow
  \lfloor (i-T)/N \rfloor$; in particular, faster than $\exp^k(i)$ for
  any $k \ge 1$. 
\end{theorem}

This growth is much more rapid than that of the
``gap principles'' of Mumford \cite{Mumford} and
Davenport-Roth \cite{DR}.
If $\cC$ is a curve of genus greater than $1$, Mumford showed
that there are constants $a,b>0$ such that
if we order the rational points of $\cC$ according to Weil height,
then the $i^{\rm th}$
point has Weil height at least $e^{a+bi}$.
In his proof, he embedded
points of $\cC$ into $\R^d$;
{\em Mumford's gap principle} roughly states that
there is a constant $C>1$ such that
if $v_1,v_2 \in \R^d$ are the
images of two points on the curve lying in a small sector, then
either $|v_1| > C \cdot |v_2|$
or $|v_2| > C \cdot |v_1|$.
Similarly, in our Theorem~\ref{thm:ml},
two indices $n_1,n_2$ lying in the 
same congruence class modulo $N$
can be considered analogous
to two vectors $v_1,v_2$ lying in a small sector.
By this analogy,
given that
Faltings \cite{Falt1}
later proved that the curve $\cC$ has only finitely many rational
points,
Theorem~\ref{thm:ml} can be viewed as evidence 
that Conjecture~\ref{nonperiodic} is true, at least
for coordinatewise maps on $(\bP^1)^g$.

In fact, in Theorem~\ref{boosting}, we will show
that the pair of
constants $(N,C)$ in the conclusion of Theorem~\ref{thm:ml}
may be replaced by the
pair $(e N, C^{e-\epsilon})$, for any positive integer $e$
and any positive real number $\epsilon>0$.
Hence, by the same analogy to
Mumford's gap principle, we prove that ``the smaller the angles''
between two indices, ``the larger the gap'' between them.

Other partial results towards Conjecture~\ref{dynamical Mordell-Lang}
may be found in \cite{Bell,new-log,Mike,Par,Jason}.
In addition,
\cite{Mike-2} discusses a generalization of Conjecture~\ref{dynamical
Mordell-Lang} for orbits of points under the action of a commutative,
finitely generated semigroup of endomorphisms of $X$, which is itself
a generalization of the classical Mordell-Lang
conjecture. Conjecture~\ref{dynamical Mordell-Lang} also fits into Zhang's
far-reaching system of dynamical conjectures (see
\cite{ZhangLec}). Zhang's conjectures include dynamical analogues of the
Manin-Mumford and Bogomolov conjectures for abelian varieties (now
theorems of \cite{Raynaud1, Raynaud2}, \cite{Ullmo} and
\cite{Zhang}).

Our proof of Theorem~\ref{thm:ml} uses $p$-adic dynamics.
First we find a suitable prime number $p$ such that
$V$, $\Phi$, and $P$ are defined over $\Q_p$, and $\Phi$ has good
reduction modulo $p$. Then, using work of Rivera-Letelier \cite{Riv1},
we carefully choose a positive integer $N$, and for each
$\ell=0,\dots,N-1$, we construct finitely many
multivariable $p$-adic power series $G_{H,\ell}(z_0,z_1,\dots,z_m)$
such that for $n$ sufficiently large,
we have $\Phi^{\ell + nN}(P)\in V$ if and only if
$G_{H,\ell}(n,p^n,p^{2^n},\dots,p^{m^n})=0$ for all $H$.
We then show that either $G_{H,\ell}$ is
identically zero for all $H$
(which implies conclusion~(i) of Theorem~\ref{thm:ml}),
or the integers $n$ with
$\Phi^{\ell + nN}(P)\in V(\C)$ grow as in conclusion (ii). 

For each prime number $p$, we also construct an example (see
Proposition~\ref{prop:badex}) of a power series $f\in\Zp[[z]]$
such that for an infinite increasing sequence
$\{n_k\}_{k\ge 1}\subset \N$
we have $f(p^{n_k})=n_k$, and moreover $n_{k+1} < n_k+ p^{2n_k}$
for each $k\ge 1$.  This example shows
that  Theorem~\ref{thm:ml} cannot
be improved merely by sharpening our $p$-adic methods;
some new technique would be required for
a full proof of Conjecture~\ref{nonperiodic}.

The outline of the paper is as follows.  In Section~\ref{p-adic analysis}
we present some Lemmas from $p$-adic dynamics, and in
Section~\ref{sect:speeds} we state and prove a technical Lemma on the
rapid growth of integer solutions to certain $p$-adic functions.  In
Section~\ref{proof of main result} we prove Theorem~\ref{thm:ml}, in
Section~\ref{proof of counting} we prove
Theorem~\ref{thm:main-counting}, and in Section~\ref{sec:control} we
prove Theorem~\ref{thm:controlling-c-n-for-curves}. Finally, in
Section~\ref{analytic example} we prove Proposition~\ref{prop:badex},
which shows that our Theorem~\ref{thm:ml} cannot be improved through
purely $p$-adic analytic methods.

{\em Acknowledgements:}
R.B. gratefully acknowledges the support of NSF Grant DMS-0600878.
Research of D.G. was partially supported by a grant from the NSERC.
P.K. was partially supported by grants from the G\"oran
Gustafsson Foundation, the Knut and Alice Wallenberg foundation, the
Royal Swedish Academy of Sciences, and the
Swedish Research Council.
T.T. was partially supported by NSA
    Grant 06G-067 and NSF Grant DMS-0801072.

\section{Background on $p$-adic dynamics}
\label{p-adic analysis}

Fix a prime $p$.  As usual, $\Zp$ will denote the ring of
$p$-adic integers, $\Qp$ will denote the field of $p$-adic
rationals, and $\Cp$ will denote the completion of an algebraic
closure of $\Qp$.
Given a point $y\in\Cp$ and a real number $r>0$, write
$$D(y,r) = \{x\in\Cp : |x-y|_p < r\}, \quad
\Dbar(y,r) = \{x\in\Cp : |x-y|_p \leq r\}$$
for the open and closed disks, respectively, of radius $r$
about $y$ in $\Cp$.

Write $[y]\subseteq \PCp$ for the residue class of a point $y\in\PCp$.
That is, $[y]=D(y,1)$ if $|y|\leq 1$,
or else $[y]=\PCp\setminus\Dbar(0,1)$ if $|y|>1$.

The action of a $p$-adic power series
$f\in\Zp[[z]]$ on $D(0,1)$ 
is either attracting
(i.e., $f$ contracts distances)
or quasiperiodic
(i.e., $f$ is distance-preserving),
depending on its linear coefficient.
Rivera-Letelier
gives a more precise description of this dichotomy
in \cite[Sections~3.1 and~3.2]{Riv1}.
The following two Lemmas essentially reproduce
his Propositions~3.3 and~3.16,
but we also verify that 
the power series he
defines also have coefficients in $\Qp$, not just in $\Cp$.

\begin{lemma}
\label{lem:attr}
Let $f(z)=a_0+a_1 z + a_2 z^2 + \cdots \in\Zp[[z]]$ be a
nonconstant power series with $|a_0|_p, |a_1|_p < 1$.
Then there is a point $y\in p\Zp$ such that
$f(y)=y$, and  $\lim_{n\rightarrow\infty}f^n(z)=y$
for all $z\in D(0,1)$.
Write $\lambda=f'(y)$; then
$|\lambda|_p < 1$,
and:

\begin{enumerate}
\item (Attracting).
If $\lambda\neq 0$, then there is a radius $0<r<1$ and
a power series $u\in\Qp[[z]]$ mapping
$\Dbar(0,r)$ bijectively onto $\Dbar(y,r)$
with $u(0)=y$, such that 
for all $z\in D(y,r)$ and $n\geq 0$,
$$f^n(z) = u(\lambda^n u^{-1}(z)).$$

\item (Superattracting).
If $\lambda=0$, then write $f$ as
$$f(z) = y + c_m (z-y)^m + c_{m+1} (z-y)^{m+1} + \cdots \in\Zp[[z-y]]$$
with $m\geq 2$ and $c_m\neq 0$.
If $c_m$ has an $(m-1)$-st root in $\Zp$, then 
there are radii 
$0<r,s < 1$ and a power series $u\in\Qp[[z]]$ mapping
$\Dbar(0,s)$ bijectively onto $\Dbar(y,r)$
with $u(0)=y$, such that
for all $z\in D(y,r)$ and $n\geq 0$,
$$f^n(z) = u\Big( (u^{-1}(z))^{m^n}\Big).$$
\end{enumerate}
\end{lemma}

\begin{proof}
Applying the Weierstrass Preparation Theorem to $f(z)-z$
(or equivalently, by inspection of the Newton polygon),
$f$ has a $\Qp$-rational fixed point $y\in D(0,1)$;
that is, $y\in p\Zp$.
Clearly $\lambda=f'(y)$ is also in $p\Zp$.
Replacing $f(z)$ by $f(z+y)-y$ (and, ultimately, replacing
$u(z)$ by $u(z)+y$), we may assume hereafter that $y=0$.
By \cite[Proposition~3.2(i)]{Riv1},
$\lim_{n\rightarrow\infty}f^n(z)=0$
for all $z\in D(0,1)$.

If $\lambda\neq 0$,
then Rivera-Letelier
defines $u^{-1}(z) := \lim_{n\rightarrow\infty} \lambda^{-n} f^n(z)$
and proves in \cite[Proposition~3.3(i)]{Riv1}
that it has an inverse $u(z)$ under composition
that satisfies the desired properties
for some radius $0<r<1$.
Note that $f \in \Qp[[z]]$, and hence
$\lambda^{-n} f^n \in \Qp[[z]]$ for all $n\geq 1$.
Thus, $u^{-1} \in \Qp[[z]]$, and therefore
$u \in \Qp[[z]]$ as well.

If $\lambda=0$, then choose
$\gamma\in\Zp\smallsetminus\{0\}$ with
$\gamma^{m-1}=c_m$, according to the hypotheses.
Define $\tilde{f}(z) := \gamma f(\gamma^{-1} z)$, so that
$\tilde{f}(z) = z^m(1+g(z))$, with $g\in z\Qp[[z]]$.
Rivera-Letelier defines
$$h(z) := \sum_{n\geq 0} m^{-n-1} \log\Big(1 + g(\tilde{f}^n(z))\Big)
\in z \Qp[[z]]$$
in \cite[Proposition~3.3(ii)]{Riv1},
where $\log(1+z) = z -z^2/2 + z^3/3 - \cdots$.
He then sets 
$\tilde{u}^{-1}(z) := z\exp(h(z))$, where
$\exp(z) = 1 + z + z^2/2! + \cdots$,
and shows that the inverse $\tilde{u}$ of
$\tilde{u}^{-1}$ has all the
desired properties for $\tilde{f}$;
note also that $\tilde{u}\in\Qp[[z]]$, because
$\log(1+\cdot)$, $\exp$, $g$, $\tilde{f}\in\Qp[[z]]$.
Hence, $u(z) = \gamma^{-1} \tilde{u}(z)\in\Qp[[z]]$
has the desired properties for $f$, mapping
some disk $\Dbar(0,s)$ bijectively onto some disk
$\Dbar(y,r)\subseteq D(0,1)$.
Finally, the radius $s$
must be less than $1$, or else $u(1)\neq y$ will be
fixed by $f$, contradicting the fact that
$\lim_{n\to \infty} f^n(u(1))=y$.
\end{proof}

\begin{lemma}
\label{lem:siegel}
Let $f(z)=a_0+a_1 z + a_2 z^2 + \cdots \in\Zp[[z]]$ be a
nonconstant power series with $|a_0|_p<1$ but $|a_1|_p = 1$.
Then for any nonperiodic $x\in p\Zp$,
there are: an integer $k\geq 1$,
radii $0<r< 1$ and $s\geq |k|_p$,
and a power series $u\in\Qp[[z]]$ 
mapping $\Dbar(0,s)$ bijectively onto $\Dbar(x,r)$
with $u(0)=x$,
such that for all $z\in \Dbar(x,r)$ and $n\geq 0$,
$$f^{nk}(z) = u(nk+ u^{-1}(z)).$$
\end{lemma}

\begin{proof}
  Because $f\in\Zp[[z]]$ with $|c_1|_p=1$ and $|c_0|_p<1$, $f$ maps
  $D(0,1)$ bijectively onto itself.  Therefore, by
  \cite[Corollaire~3.12]{Riv1}, $f$ is quasiperiodic, which means in
  particular that for some $0<r<1$ and for some positive integer $k$,
  the function
$$f_*(z):=\lim_{|n|_p\to 0}\frac{f^{nk}(z) - z}{nk}$$
converges uniformly on $\Dbar(x,r)$ to a power series
in $\Cp[[z-x]]$.
In fact, $f_*\in\Qp[[z-x]]$, because
$(f^{nk}(z) - z)/(nk)\in\Qp[[z-x]]$ for every $n$.

Since $x$ is not periodic, $f_*(x)\neq 0$,
by \cite[Proposition~3.16(1)]{Riv1}.
Define $u^{-1}\in\Qp[[z-x]]$ to be the antiderivative
of $1/f_*$ with $u^{-1}(x)=0$.  Because $(u^{-1})'(x)\neq 0$,
we may decrease $r$
so that $u^{-1}$ is one-to-one on $\Dbar(x,r)$.
Also replace $k$ by a multiple of itself
so that $f^k(x)\in\Dbar(0,r)$, and
write $\Dbar(0,s):=u^{-1}(\Dbar(x,r))$.
The proof of \cite[Proposition~3.16(2)]{Riv1} shows
that the inverse $u$ of $u^{-1}$, which must also
have coefficients in $\Qp$, satisfies
the desired properties.
\end{proof}

\begin{remark}
In fact, the integer $k$ in Lemma~\ref{lem:siegel} is
at most $p$, at least in the case that $p>3$; see
Proposition~\ref{thm:bounding-k}.
\end{remark}

\section{A growth lemma}
\label{sect:speeds}

We will also need a technical lemma about the growth
of certain solutions
of multivariate $p$-adic power series.  Before stating it, we
set some notation.  First, we fix $m\geq 1$,
and with $\bN=\{0,1,2,\ldots\}$ denoting the
natural numbers, we order $\bN^m$ by lexicographic ordering reading
right-to-left.  That is, $(b_1,\ldots,b_m) \prec (b'_1,\ldots,b'_m)$
if either $b_m<b'_m$, or $b_m=b'_m$ but $b_{m-1}<b'_{m-1}$,
or $b_m=b'_m$ and $b_{m-1}=b'_{m-1}$ but $b_{m-2}<b'_{m-2}$, etc.
Note that this order $\prec$ gives a well-ordering of $\bN^m$.

Next, for any $m$-tuple $w=(a,b_2,\ldots,b_m)\in\bN^m$ and $n\geq 0$,
set
$|w|:=a+b_2+\cdots +b_m$, and define the function $f_w:\bN\to\bN$ by
\begin{equation}
\label{eq:fwdef}
f_w(n) = a n + \sum_{j=2}^m b_j j^n.
\end{equation}
For any $w,w'\in\bN^m$, note that
$w\prec w'$ if and only if $f_w(n)$ grows more slowly than $f_{w'}(n)$ as
$n\to\infty$.
Finally, given a power series
$G\in\Qp[[z_0,z_1,\ldots,z_m]]$, we may write $G$ uniquely as
\begin{equation}
\label{eq:Gsumdef}
G(z_0,z_1,\ldots,z_m)=
\sum_{w\in\bN^m} g_w(z_0) z^w,
\end{equation}
where $g_w\in\Q_p[[z_0]]$, and for $w=(a,b_2,\ldots,b_m)\in \bN^m$,
$z^w$ denotes
$$z^w = z_1^a z_2^{b_2} z_3^{b_3} \cdots z_m^{b_m}.$$

Armed with this notation, we can now state our Lemma.

\begin{lemma}
\label{lem:speeds}
Let
$G(z_0,z_1, z_2, \ldots,z_m)\in \Qp[[z_0,z_1,z_2, \ldots,z_m]]$ be 
a nontrivial
power series in $m+1\geq 1$ variables.
Write $G=\sum_w g_w(z_0) z^w$ as in
equation~\eqref{eq:Gsumdef}, and
let $v\in\bN^m $ be the minimal
index with respect to $\prec$ such that $g_v\neq 0$.
Assume that $g_v$ converges on $\Dbar(0,1)$
and that there is a radius $0<s\leq 1$ such that
for all $\alpha\in\Zp$, $g_v$ does not vanish
at more than one point of the disk $\Dbar(\alpha,s)$.
Assume also that there exists $B > 0$ such that for each $w\succ v$,
all coefficients of $g_w$ have absolute value at most $p^{B|w|}$.

Then there exists $C>1$
with the following property:
If $\alpha\in\Dbar(0,1)$, and if
$\{n_i\}_{i\geq 1}$
is a strictly increasing sequence 
of positive integers such that for each $i\geq 1$,
\begin{enumerate}
\item[$(a)$] $|n_i - \alpha|_p \leq s$, and

\item[$(b)$] 
$G(n_i,p^{n_i},p^{2^{n_i}},p^{3^{n_i}},...,p^{m^{n_i}}) =   0$,
%
\end{enumerate}
then 
$n_{i+1} - n_i > C^{n_i}$ for all sufficiently large $i$.
\end{lemma}

\begin{proof}
If $g_w=0$ for all $w\neq v$, then $G=g_v(z_0) z^{v}$.
By hypothesis (b), then,
the one-variable nonzero power series
$g_{v}(z_0)$ vanishes at all points of the sequence
$\{n_i\}_{i\geq 1}$, a contradiction;
hence, no such sequence exists.
(In particular, if $m=0$, then
$G$ is a nontrivial power series in the one variable $z_0$,
and therefore $G$ vanishes at only finitely many points $n_i$.)
Thus, we may assume that $g_w$ is nonzero for some $w\succ v$.

\begin{claim}
\label{key technical claim}
For any $B> 0$,
there is an integer $M = M(v, B)\geq 0$
such that for each $w \succ v$ and $n\geq M$,
$$f_w(n) - f_{v}(n) \geq  n + B (|w| - |v| - 1).$$
\end{claim}

\begin{proof}[Proof of Claim~\ref{key technical claim}.]
Write $v = (a, b_2,...,b_m)$, and choose
$M\geq B$ large enough
so that $j^M \geq (a+1)M + \sum_{k=2}^{j-1} b_k k^M$ for all $j=2,\ldots, m$.
Write $w = (a', b'_2,..,b'_m)$. Then
$$f_w(n) - f_{v}(n) = 
(a' - a)n + (b'_2 - b_2) 2^n + ... + (b'_m - b_m) m^n.$$
We consider two cases:
 
{\bf Case 1.}  If $b'_k = b_k$ for each $k=2,..,m$, then
$a'>a$, and therefore
$$ f_w(n) - f_{v}(n) -n  = (a' - a -1)n
\geq (a'-a -1)B = B(|w| - |v| - 1)$$
for $n \ge M$, because $M\geq B$.
 
{\bf Case 2.}  Otherwise, there exists $k = 2,..,m$ such that $b'_k > b_k$.
Let $j$ be the largest such $k$, so that $b'_k = b_k$ for $k> j$. Then
\begin{multline*}
 f_w(n) - f_{v}(n) - B |w| + B |v| \\
 =  (a' - a) (n - B) +
 \sum_{k=2}^{j-1} (b'_k - b_k) (k^n - B) + (b'_j - b_j) (j^n - B)\\
 \ge  - a n - \sum_{k=2}^{j-1} b_k k^n + j^n  - B  \ge n-B,
\end{multline*}
where the first inequality is because $n\geq B$
and  $b'_j - b_j\geq 1$,
and the second is because $n\geq M$.
This finishes the proof of Claim~\ref{key technical claim}.
\end{proof}

By hypothesis~(b), for any $i$ such that $n_i\geq M(v,B)$,
\begin{equation}
\label{eq:gvbound}
|g_v(n_i)|_p 
=
\bigg| 
\sum_{w\succ v} g_w(n_i) p^{f_w(n_i)-f_v(n_i)}
\bigg|_p 
\leq  p^{-n_i+ B|v| + B},
\end{equation}
where the inequality is
by Claim~\ref{key technical claim}, 
the fact that $|n_i|_p\le 1$,
and the fact that the absolute values of all coefficients
of $g_w$ are at most $p^{B|w|}$.
Let $\beta\in\Dbar(\alpha,s)\cap\Zp$ be a limit point
of the sequence $\{n_i\}_{i\geq 1}$.  Then by
inequality~\eqref{eq:gvbound}, we have $g_v(\beta)=0$.
Thus, $g_v$ can be written as
$$
g_v(z) = \sum_{i\geq \delta} c_i (z-\beta)^i,
$$
where $\delta\geq 1$ and $c_{\delta}\neq 0$.
In fact, we must have
$|c_{\delta}|_p s^{\delta} > |c_i|_p s^i$ for all $i>\delta$;
otherwise, inspection of the Newton polygon
shows that $g_v$ would have a zero besides $\beta$
in $\Dbar(\alpha,s)$.
Thus,
for $i$ sufficiently large (i.e., such that $n_i\geq M(v,B)$), we have
$$
|c_\delta(n_i-\beta)^\delta|_p = |g_v(n_i)|_p
\leq p^{-n_i+O(1)},
$$
by hypothesis~(a) and inequality~\eqref{eq:gvbound}, and hence 
\begin{equation}
  \label{eq:ni-minus-beta-is-small}
|n_i-\beta|_p \leq |c_{\delta}|_p^{-1/\delta} p^{-n_i/\delta+O(1)}.
\end{equation}
It follows that
\begin{equation}
  \label{eq:n-sub-i-congruence}
n_{i+1}  \equiv  n_i \mod p^{\lfloor n_i/\delta-O(1) \rfloor}.
\end{equation}
Hence, if we choose $C$ such that $1<C < p^{1/\delta}$, we have
$n_{i+1} - n_i > C^{n_i}$
for $i$ sufficiently large, as desired.
\end{proof}

\begin{remark}
Lemma~\ref{lem:speeds} holds also if $G$ is defined over a finite
extension $K$ of $\Qp$; the only significant change is that the
constant $C$ will depend also on the ramification index $e$ of
$K/\Qp$.
\end{remark}

\section{Proof of Theorem~\ref{thm:ml}}
\label{proof of main result}

%
\begin{proof}
If any $x_j$ is preperiodic under $f_j$
(without loss, $x'_g:=f_g^m(x_g)$ is periodic under $f_g$),
consider the action of
$\Phi' :=(f_1,\ldots,f_{g-1})$ on $(\bP^1)^{g-1}$, with
$V' :=V\cap\{z_g=x'_g\}$ viewed as a subvariety of $(\bP^1)^{g-1}$.
By this reduction,
we may assume without loss that no $x_j$ is preperiodic.

{\bf Step (i).}  Our first goal is to find an appropriate prime $p$
so that we may work over $\bZ_p$.

Choose homogeneous coordinates for each $\bP^1$, so that
we may write each $f_j$ as $f_j([a:b]) = [\phi_j(a,b): \psi_j(a,b)]$ for
homogeneous relatively prime polynomials $\phi_j, \psi_j \in \bC[a,b]$;
write $P$ in these coordinates as well.  
Let $\cV$ be a finite set of polynomials (in $g$ pairs of
homogeneous variables)
generating the vanishing ideal of the variety $V$.
Let $R$ be the subring of
$\bC$ generated by the coordinates of $P$, the coefficients of the
$\phi_j$ and $\psi_j$, the coefficients of each polynomial
$H\in\cV$, and the resultants
  $$\Res(\phi_1,\psi_1), \ldots, \Res(\phi_g,\psi_g)$$
along with their reciprocals
  $$1/(\Res(\phi_1,\psi_1)), \ldots, 1/(\Res(\phi_g,\psi_g)).$$

Each $f_j$ has only finitely many critical points, and therefore
only finitely many superattracting periodic points, because any
superattracting cycle must include a critical point.  Adjoin
the coordinates of all the superattracting periodic points to $R$.
For each
superattracting periodic point $Q$ of $f_j$ of period $\kappa$,
choose a local coordinate $x_Q$ at $Q$ defined over the fraction
field $K_R$ of $R$, and write $f_j^{\kappa}$
as $f_j^{\kappa}(x_Q) = c_m x_Q^m + O(x_Q^{m+1})$, where $m\geq 2$
and $c_m\in K_R^{\times}$.  Now adjoin an $(m-1)$-st root of $c_m$
to $R$ at each such point.
    
By \cite[Lemma 3.1]{Bell} (see also \cite{Lech}), we may find an
embedding of $R$ into $\bZ_p$ for some prime $p$.  Thus, we may
consider $P$ to be a point in $(\bP^1)^g(\bZ_p)$ and $\Phi$
to be defined over $\Zp$.
Because the
resultants $\Res(\phi_j,\psi_j)$ are all mapped to units in $\bZ_p$,
each map $f_j$ has good reduction, i.e.,
reducing $f_j$ modulo $p$ gives an
endomorphism of $\bP^1$ defined over $\bF_p$.
The fact that each $c_m^{1/(m-1)}$ is in $\Qp$ will be needed
to invoke Lemma~\ref{lem:attr}(ii) in Step (ii).

{\bf Step (ii).}
Next, we will apply Lemmas~\ref{lem:attr} and~\ref{lem:siegel}
to produce power series $u_j$ to be used later in the proof,
and to find integers $k_j$ that will be used to construct
the integer $N$.

Write $P:=(x_1,\dots,x_g)\in(\bP^1)^g(\mathbb{Z}_p)$.
There are only $p+1$ residue classes in $\bP^1(\bZ_p)$;
hence, for each $j=1,\ldots, g$, there are integers $k_j\geq 1$
and $\ell_j\geq 0$ with $k_j+\ell_j \leq p+1$ such that
$f_j^{k_j}$ maps the residue class $[f_j^{\ell_j}(x_j)]$ into itself.
By a $PGL(2,\Zp)$-change of coordinates at each $j$,
we may assume that $f_j^{\ell_j}(x_j)\in p\Zp$, and therefore
$f_j^{k_j}$ may be written as a nonconstant power
series in $\Zp[[z]]$ mapping $D(0,1)$ to itself.

If $|(f_j^{k_j})'(f^{\ell_j}(x_j))|_p<1$
(i.e., attracting or superattracting),
we may apply Lemma~\ref{lem:attr}.
(In the superattracting case we are using
the fact that the corresponding coefficient $c_m$
has an $(m-1)$-st root in $\Qp$.  Although the
coordinate $\tilde{x}_Q$ we may now be using
at the superattracting point may differ from
the coordinate $x_Q$ of Step~(i), both are
defined over $\Qp$.  Thus, there is some $\alpha\in\Qp^{\times}$
such that $x_Q = \alpha \tilde{x}_Q + O(\tilde{x}_Q^2)$, and
the expansion $c_m x_Q^m + O(x_Q^{m+1})$
becomes $\alpha^{m-1} c_m \tilde{x}_Q^m + O(\tilde{x}_Q^{m+1})$.
Hence, the integer $m$ is preserved, and the
lead coefficient still has an $(m-1)$-st root in $\Qp$.
Of course, this root is in fact in $\Zp$, because our
choice of coordinates forced  $f_j^{k_j}\in\Zp[[z]]$.)
The Lemma yields that there is a
point $y_j\in D(0,1)$ fixed by $f_j^{k_j}$
with multiplier $\lambda_j$,
along with radii $r_j$ and $s_j$ (where $s_j := r_j$
in the non-superattracting case) and associated
power series $u_j\in\Qp[[z]]$.
Increase $\ell_j$
if necessary so that $f_j^{\ell_j}(x_j)\in\Dbar(y_j,r_j)$.
Define $\mu_j:=u_j^{-1}(f^{\ell_j}(x_j))$; note that
$\mu_j\in p\Zp$, because $s_j<1$.  In addition,
$\mu_j\neq 0$, because $u_j$ is bijective
and $u_j(0)=y_j$ is fixed by $f_j$,
while $u_j(\mu_j)=f_j^{\ell_j}(x_j)$ is not.

If $|(f_j^{k_j})'(f^{\ell_j}(x_j))|_p=1$ (i.e., quasiperiodic),
apply Lemma~\ref{lem:siegel} to $f_j^{k_j}$ and the point
$f^{\ell_j}(x_j)$ to
obtain radii $r_j$ and $s_j$ and a power series
$u_j$.  Set $\mu_j:=u_j^{-1}(f^{\ell_j}(x_j))$.
In addition,
multiply $k_j$ by the $k$ from Lemma~\ref{lem:siegel}
so that $f_j^{k_j + \ell_j}(x_j)\in\Dbar(f^{\ell_j}(x_j),r_j)$.

{\bf Step (iii).}
In this step, we consider only the
attracting or superattracting cases
($|(f_j^{k_j})'(f^{\ell_j}(x_j))|_p<1$).
We will show that certain functions of $n$ can be
expressed as power series in $n$, $p^n$, and $p^{m^n}$.

If $\lambda_j\neq 0$ (i.e. $0<|\lambda_j|_p<1$),
write $\lambda_j = \beta_j p^{e_j}$, where $e_j\geq 1$
and $\beta_j\in\Zp^{\times}$.
If $\beta_j$ is a root of unity, we can
choose $M\geq 1$ such that $\beta_j^M=1$.
If $\beta_j$ is not a root of unity,
it is well known that
there is an integer $M\geq 1$ such
that $\beta_j^{n M}$
can be written as a power series in $n$
with coefficients in $\Z_p$.
(For example,
apply Lemma~\ref{lem:siegel} to the function
$z\mapsto\beta_j z$ and the point $p$.
In fact, by Theorem~\ref{thm:analytic-congruent-to-z},
we can choose $M$ to be the smallest positive integer
such that $|\beta_j^{M} - 1|_p < 1$, so that
$M|(p-1)$.)
Either way, then, replacing
$k_j$ by $Mk_j$ (and hence
$\lambda_j$ by $\lambda_j^M$,  $e_j$ by $Me_j$,
and $\beta_j$ by $\beta_j^M$), we can write
\begin{equation}
\label{eq:gdef1}
\lambda_j^n = (p^n)^{e_j} g_j(n)
\qquad \text{for all integers } n\geq 0,
\end{equation}
for some power series $g_j(z)\in\Zp[[z]]$.

If $\lambda_j=0$ (i.e. $y_j$ is a superattracting
fixed point of $f^{k_j}$), let $m_j\geq 2$ be the corresponding
integer from Lemma~\ref{lem:attr}(ii), and write $m_j=ap^b$,
for integers $a\geq 1$ and $b\geq 0$, with $p\nmid a$.
Then as in the previous paragraph, $m_j^n$ can be written
as a power series in $n$ and $p^n$ with coefficients in $\Zp$,
after replacing $k_j$ by a multiple $Mk_j$ (and hence $m_j$ by $m_j^M$).


In addition, recall that $\mu_j=u_j^{-1}(f^{\ell_j}(x_j))$
satisfies $0<|\mu_j|_p<1$; thus,
we can write $\mu_j=\beta_j p^{e_j}$,
where $e_j\geq 1$ and $\beta_j\in\Zp^{\times}$.
If $\beta_j$ is a root of unity with, say, $\beta_j^{M'}=1$
for some $M'\geq 1$, choose an integer $M\geq 1$
so that $M' | (m_j^{2M} - m_j^M)$.
Then replacing $k_j$ by $Mk_j$ (and hence $m_j$ by $m_j^M$),
$\beta_j^{m_j^{n}}=\beta_j^{m_j}$ is constant in $n$.

On the other hand, if $\beta_j$ is not a root of
unity, then as in the attracting case
there is an integer $1\leq M'\leq p-1$ such that $\beta_j^{n M'}$
can be written as a power series in $n$ over $\Z_p$.
Again, choose $M\geq 1$ such that $M' | (m_j^{2M} - m_j^M)$,
and replace $k_j$ by $Mk_j$ (and $m_j$ by $m_j^M$).
Then $\beta_j^{m_j^{n}}$ can be written as a power series
in $(m_j^{n}-m_j)/M'$ with coefficients in $\Zp$;
since $p\nmid M'$, and expressing
$m_j^n$ as a power series,
$\beta_j^{m_j^{n}}$ can in fact
be written as a power series in $n$ and $p^n$,
with coefficients in $\Zp$.

Thus, in the superattracting case, after having 
increased $k_j$ if necessary, we can write
\begin{equation}
\label{eq:gdef2}
\mu_j^{m_j^n} = (p^{m_j^n})^{e_j} g_j(n,p^n)
\qquad \text{for all integers } n\geq 0,
\end{equation}
for some power series $g_j(z_0,z_1)\in\Zp[[z_0,z_1]]$.

{\bf Step (iv).}
Let $k:=\lcm(k_1,\ldots,k_g)\geq 1$; this will essentially
be our value of $N$ in the statement of Theorem~\ref{thm:ml} except
for one more change in Step~(vi).  Note that replacing
each $k_j$ by $k$ does not change the conclusions of
Steps~(ii) and~(iii).
Specifically, the radii $r_j$, $s_j$, the power series $u_j$,
the integer $\ell_j$, and the point $\mu_j$ were chosen before
changing $k_j$, and so they are unaffected.
Moreover, in the quasiperiodic case,
$f_j^{k_j + \ell_j}(x_j)$ still lies in $\Dbar(f^{\ell_j}(x_j),r_j)$.
In the attracting case,
$\lambda_j$ is replaced by a power of itself,
and $\lambda_j^n$
can still be written as in \eqref{eq:gdef1}.
In the superattracting case,
$m_j$ is replaced by a power of itself, and
$\mu_j^{m_j^n}$ can still be written as in \eqref{eq:gdef2}.
In addition, given our changes to $\lambda_j$
and $m_j$ in the attracting and superattracting cases,
the functional equations of Lemmas~\ref{lem:attr}
and~\ref{lem:siegel} are preserved under this change of $k$.

Let $L=\max\{\ell_1,\ldots,\ell_g\}$.
For each $\ell=L,\ldots,L+k-1$ and each $j=1,\ldots,g$,
choose a linear fractional transformation
$\eta_{j,\ell}\in PGL(2,\Zp)$ so that
$\eta_{j,\ell}\circ f_j^{\ell}(x_j)\in D(0,1)$.
Hence, $\eta_{j,\ell}\circ f_j^{\ell-\ell_j}(D(0,1))\subseteq D(0,1)$,
because $f_j$ has good reduction.
Finally, define $E_{j,\ell}=\eta_{j,\ell}\circ f_j^{\ell-\ell_j}\circ u_j$,
so that $E_{j,\ell}\in\Qp[[z]]$ maps $\Dbar(0,s_j)$
into $D(0,1)$.

{\bf Step (v).}
For each $\ell=L,\ldots,L+k-1$
and each $f_j$ of attracting but not superattracting type,
define the power series
$$F_{j,\ell}(z_0,z_1) =
E_{j,\ell}\big(z_1^{e_j} g_j(z_0) \mu_j\big)
\in \bQ_p[[z_0,z_1]],$$
where $E_{j,\ell}$ is as in Step~(iv),
$g_j$ is as in equation~\eqref{eq:gdef1},
and $\mu_j = u_j^{-1}(f^{\ell_j}(x_j))$ as before,
so that
$F_{j,\ell}(n,p^n)=\eta_{j,\ell}\circ f_j^{\ell + nk}(x_j)$
for all $n\geq 0$.

Because $E_{j,\ell}$ maps $\Dbar(0,s_j)$ into $D(0,1)$,
there is some $B> 0$ such that for every $i\geq 0$,
the coefficient of $z^i$
in $E_{j,\ell}(z)$ has absolute value at most $p^{Bi}$.
Recalling also that $g_j\in\Z_p[[z]]$ and $|\mu_j|_p<1$,
it follows that if we write
$F_{j,\ell}(z_0,z_1)=\sum_{i=0}^{\infty} h_i(z_0)z_1^i$
(where $h_i\in\Qp[[z]]$),
then for each $i\geq 0$, all coefficients of $h_i$ have absolute value
at most $p^{Bi}$.

For each $\ell=L,\ldots,L+k-1$
and each $f_j$ of superattracting type, define the power series
$$F_{j,\ell}(z_0,z_1,z_{m_j}) =
E_{j,\ell}\big( g_j(z_0,z_1) z_{m_j}^{e_j} \big)
\in \bQ_p[[z_0,z_1,z_{m_j}]],$$
where $E_{j,\ell}$ is as in Step~(iv) and
$g_j$ is as in equation~\eqref{eq:gdef2},
so that
$F_{j,\ell}(n,p^n,p^{m_j^{n}})=\eta_{j,\ell}\circ f_j^{\ell + nk}(x_j)$
for all $n\geq 0$.

Again, because $E_{j,\ell}$ maps $\Dbar(0,s_j)$ into $D(0,1)$,
there is some $B> 0$ such that the coefficient of $z^i$
in $E_{j,\ell}(z)$ has absolute value at most $p^{Bi}$.
Hence, if we write
$F_{j,\ell}(z_0,z_1,z_{m_j})
=\sum_{i_1,i_2\geq 0} h_{i_1,i_2}(z_0)z_1^{i_1}z_{m_j}^{i_2}$
(where $h_{i_1,i_2}\in\Qp[[z]]$),
then as before, since $g_j\in\Zp[[z_0,z_1]]$,
all coefficients of $h_{i_1,i_2}$ have absolute value
at most $p^{Bi_2}$, and hence at most $p^{B(i_1 + i_2)}$.

Finally, if $f_j$ is of quasiperiodic type, then
define, for each $\ell=L,\ldots,L+k-1$, the power series
$$F_{j,\ell}(z_0) = E_{j,\ell}(k z_0 + \mu_j )
\in \bQ_p[[z_0]],$$
so that $F_{j,\ell}(n)=\eta_{j,\ell}\circ f_j^{\ell + nk}(x_j)$
for all $n\geq 0$.
All coefficients of $F_{j,\ell}$ are at most $1$,
because $|k|_p,|\mu_j|_p\leq s_j$
and $E_{j,\ell}$ maps $\Dbar(0,s_j)$ into $D(0,1)$.

{\bf Step (vi).}
Let $m$ be the maximum of the $m_j$ across all the
superattracting $f_j$'s, or $m=1$ if there are none.
For each $\ell=L,\ldots,L+k-1$,
let $\cV_{\ell}\subseteq\Zp[t_1,\ldots,t_g]$ be
the finite set of polynomials $\cV$
generating the vanishing ideal of $V$ from step~(i),
but now dehomogenized with respect to the coordinates
determined by
$(\eta_{1,\ell},\ldots,\eta_{g,\ell})$.
For each polynomial $H\in\cV_{\ell}$, define
$$G_{H,\ell}(z_0,\ldots,z_m) = H(F_{1,\ell},F_{2,\ell},\ldots,F_{g,\ell})
\in \bQ_p[[z_0,z_1,\ldots,z_m]].$$
Then by construction,
$G_{H,\ell}(n,p^n,p^{2^n},\ldots,p^{m^n})$
is defined for all integers $n\geq 0$, and is zero
precisely at those $n$ for which $\Phi^{\ell+nk}(P)\in V$.

For each nontrivial $G_{H,\ell}$, write
$$G_{H,\ell}(z_0,p^n,p^{2^n},\ldots,p^{m^n})
=\sum_{w\in\bN^m} g_w(z_0)p^{f_w(n)}$$
and select $v\in\bN^m$ as in the statement of Lemma~\ref{lem:speeds}.
By choosing the maximum of all the (finitely many) bounds $B$
from Step~(v), and because all coefficients of $H$ lie
in $\Zp$, all coefficients of $g_w$ are at most $p^{B|w|}$,
for every $w\in\bN^m$.
Since $G_{H,\ell}(n,p^n,p^{2^n},\ldots,p^{m^n})$
is defined at every $n\geq 0$, $g_v$ must converge
on $\Dbar(0,1)$.  By the Weierstrass Preparation Theorem (or, equivalently,
from the Newton polygon), $g_v$ has only finitely many zeros in $\Dbar(0,1)$.
Let $0<s_{H,\ell}\leq 1$ be the minimum distance
between any two distinct such zeros, and
let $s$ be the minimum of all the
radii $s_{H,\ell}$ across all such pairs $(H,\ell)$.
The set $\Zp$ may be covered by the disks
$D(0,s)$, $D(1,s)$, $\ldots$, $D(p^M-1,s)$,
for some integer $M\geq 0$.
Let $N:=p^M k$.

Apply Lemma~\ref{lem:speeds}
(with the values of $B$ and $s$ from the previous paragraph)
to every nontrivial $G_{H,\ell}$, and
let $C_0>1$ be the minimum of the resulting constants.
Choose any $\epsilon>0$, and let $C:=C_0^{p^M-\epsilon}>1$.

{\bf Step (vii).}
Unless conclusion~(ii) of Theorem~\ref{thm:ml} holds
for these values of $C$ and $N$, there is some $\ell\in\{L,\ldots,L+N-1\}$,
and there are infinitely many pairs $(n,n')$ of positive integers
such that
\begin{enumerate}
\item $\Phi^{\ell + n N}(P), \Phi^{\ell + n' N}(P)\in V$, and
\item $0< n' -  n  \leq C^{n}$.
\end{enumerate}
For any fixed $n\geq 1$, there are only finitely
many choices of $n'$ for which condition~(ii) above holds;
thus, there are pairs $(n,n')$ with $n$
arbitrarily large satisfying these two conditions.

Write $\ell = \ell_1 + \alpha k$ for integers
$L\leq \ell_1 < L+k$
and $0\leq \alpha < p^M$.
For each pair $(n,n')$
above, write $n_1 = np^M + \alpha$
and $n'_1 = n'p^M + \alpha$.  Then there are infinitely
many pairs $(n_1,n'_1)$ such that
\begin{enumerate}
\item[(1)] $\Phi^{\ell_1+n_1k}(P), \Phi^{\ell_1+n_1'k}(P)\in V$,
\item[(2)] $n_1\equiv n_1'\equiv \alpha\pmod{p^M}$, and
\item[(3)] $0<(n_1'-n_1)/p^M \leq
  C^{(n_1 - \alpha)/p^M}$.
\end{enumerate}
Recalling that
$C=C_0^{p^M-\epsilon}>1$ and $\alpha\geq 0$,
condition~(3) becomes
\begin{enumerate}
\item[(3')] $0<n_1'-n_1 \leq p^M C_0^{(n_1-\alpha)(1-\epsilon p^{-M}))}
\leq C_0^{n_1}$,
\end{enumerate}
for $n_1$ sufficiently large
(more precisely, for $n_1\ge \frac{Mp^M\log p}{\epsilon\log C_0}$). However,
conditions~(1), (2) and (3') coupled with Lemma~\ref{lem:speeds} yield
that $G_{H,\ell_1}$ must be trivial for all $H\in\cV_{\ell_1}$.
Hence, $\Phi^{\ell_1+nk}(P)\in V$ for all $n\geq 0$.
Let $V_0$ be the Zariski closure of
$\{\Phi^{\ell_1 + nk}(P)\}_{n\geq 0}$.  Then
$\text{dim}(V_0)\geq 1$, because $V_0$ contains infinitely
many points, and $\Phi^k(V_0)\subseteq V_0\subseteq V$.
Thus, $V_0$ is a positive-dimensional periodic subvariety of $V$.
\end{proof}

In the final step of the proof, we produced the constants
$N$ and $C$ that appeared in the statement of Theorem~\ref{thm:ml}.
In fact, as the following result shows,
for any integer $e\geq 1$, we
can increase $C$ to $C^{e-\epsilon}$, at the expense of increasing
$N$ to $eN$.

\begin{theorem}
\label{boosting}
If the proof of Theorem~\ref{thm:ml}
yields constants $C>1$ and $N\ge 1$ satisfying its conclusion, then for
any integer $e>1$ and for any $\epsilon>0$, the conclusion of
Theorem~\ref{thm:ml} holds when 
replacing the pair $(C,N)$ by $(C^{e-\epsilon},e N)$.
\end{theorem}

\begin{proof}
In step~(vi) of the proof of Theorem~\ref{thm:ml}, we had
produced positive integers $k$, $L$, and $M$,
and a real constant $C_0>1$.  We then set $N=p^M k$ and
$C=C_0^{p^M-\epsilon}$.  Instead, we now set  $N:= ep^M k$
and $C:=C_0^{ep^M - \epsilon}$, as promised in the
statement of Theorem~\ref{boosting}. 

Step~(vii) of the proof of Theorem~\ref{thm:ml} still
applies even when we change all appearances of $p^M$ to $ep^M$.
More precisely, we have $0\leq \alpha < ep^M$ when we write
$\ell = \ell_1 + \alpha_k$, and we write
$n_1 = nep^M + \alpha$ and $n'_1 = n'ep^M + \alpha$.
The $\pmod{p^M}$ in condition (2) becomes $\pmod{ep^M}$, which
of course still implies congruence modulo $p^M$.  The change
from $p^M$ to $ep^M$ ultimately leaves condition~(3') as
$0<n'_1 - n_1 \leq C_0^{n_1}$, though now only for
$n_1\ge \frac{ep^M\log (ep^M)}{\epsilon\log C_0}$).
Thus, conditions (1), (2), and (3') remain the same
as before, allowing exactly the same
application of Lemma~\ref{lem:speeds}.  The rest
of the proof then goes through verbatim.
\end{proof}

\section{Proof of Theorem~\ref{thm:main-counting}}
\label{proof of counting}

\begin{proof}
Assume that $V$ does not contain a positive-dimensional periodic
subvariety.  Then there are constants $C,N,T$ as in
Theorem~\ref{thm:ml}(i).
Let $A=T + N + NC$, and
for each $\ell = T+1,\ldots,T+N$, let
$$\cS_{\ell} :=\{n>C :  \ell + nN\in\cS\}.$$
For all $i\geq 1$, let $n_{\ell,i}$ be the $i^{\rm th}$ smallest
integer in $\cS_{\ell}$, or $n_{\ell,i}=\infty$ if $|\cS_{\ell}|<i$.
Then $n_{\ell,1}>C = C\uparrow\uparrow 1$, and
$n_{\ell,i+1} > n_{\ell,i} + C^{n_{\ell,i}}
> C^{n_{\ell,i}} > C\uparrow \uparrow (i+1)$
for all $i\geq 1$.
Hence, $L_C(n_{\ell,i})\geq i$, and therefore
$L_C(M)\geq i$ for all $M\geq n_{\ell,i}$.
Summing across all $\ell$, we have
\begin{align*}
|\{ n\in\cS : n\leq M \}|
&\leq
|\{ n\in\cS : n\leq A \}| + 
\sum_{\ell=T+1}^{T+N}
|\{ n\in\cS_{\ell} : n\leq M \}|
\\
& \leq A + N \cdot L_C(M).
\qedhere
\end{align*}
\end{proof}

\section{Curves}
\label{sec:control}

The proof of Theorem~\ref{thm:controlling-c-n-for-curves} is simpler
than the proof of Theorem~\ref{thm:ml}, but it requires an
additional ingredient that is only available over number fields,
namely, the existence of a suitable indifferent cycle in at least one
of the variables (which one obtains over number fields by
\cite[Theorem~2.2]{SilSiegel} or \cite[Lemma~6.1]{Par}).  Because of
the counterexample presented in Proposition~\ref{prop:badex}, it seems
likely that a proof of Conjecture~\ref{dynamical Mordell-Lang} would
also have to involve extra information beyond what is used in the
proof of Theorem~\ref{thm:ml}.  Thus, although
Theorem~\ref{thm:controlling-c-n-for-curves} only applies to curves,
it may well be that the techniques used to prove it are better adapted
to a general proof of Conjecture~\ref{dynamical Mordell-Lang}.

To prove Theorem~\ref{thm:controlling-c-n-for-curves} we will need a
sharper version of Lemma~\ref{lem:siegel}, giving an upper bound on
$k$.
%
We first recall  the following special 
case of \cite[Theorem~3.3]{Jason}.
\begin{theorem}
\label{thm:analytic-congruent-to-z}
  Let $p>3$ be prime, let $K_{p}/\Qp$ be a finite unramified
  extension,
  and let $\mathcal{O}_{p}$ denote the ring of integers in $K_{p}$.
  Let $g(z) = a_0 + a_1 z + a_2 z^2 + \cdots \in \mathcal{O}_{p}[[z]]$
  be a power series with $|a_0|_p, |a_1-1|_p < 1$
  and for each $i\geq 2$, $|a_i|_p\leq p^{1-i}$.
  Then for any $z_0 \in \mathcal{O}_{p}$, there
  is a power series
  $u \in \cO_{p}[[z]]$ mapping $\Dbar(0,1)$ into itself
  such that $u(0) = z_0$, and $u(z+1) = g(u(z))$.
\end{theorem}
\begin{remark}
  In \cite{Jason}, the theorem is only stated for $K_{p} = \Q_p$, but
  the proof goes through essentially
  unchanged for any finite unramified extension
  of $\Q_p$.
\end{remark}

We can now give an explicit bound on $k$.  However, we give up any
claims on the size of the image of $u$.
In fact, if $z_{0}$ is a periodic point, the
map $u$ is constant.  (On the other hand, if $z_{0}$ is not periodic,
then the derivative of $u$ is nonvanishing at zero, and hence $u$ is a
local bijection.)
\begin{proposition}
\label{thm:bounding-k}
Let $p>3$ be prime,
let $K_p$ and $\cO_p$ be as in
Theorem~\ref{thm:analytic-congruent-to-z},
let $h(z) \in \mathcal{O}_{p}[[z]]$ be a power series,
and let $z_{0} \in \mathcal{O}_{p}$.
Suppose that $|h(z_0)-z_0|_p < 1$ and $|h'(z_{0})|_p=1$.  Then
there is an integer
$1\leq k \leq p^{[K_{p}:\Q_p]}$
and a power series $u\in\cO_p[[z]]$ mapping
$\Dbar(0,1)$ into $\Dbar(0,1)$
such that $u(0)=z_0$ and $h^{k}(u(z)) = u(z+1)$.
In particular,
$$
h^{nk}(z_{0}) = u(n)
\qquad\text{for all } n\geq 0.
$$
\end{proposition}
\begin{proof}
Let $q=p^{[K_{p}:\Q_p]}$ denote the cardinality of the residue field
of $\mathcal{O}_p$.   
Conjugating by a translation we may assume that $z_{0}=0$.
Let
$$
g(z) :=  h(p z)/p
= b_0 + b_1 z + b_2 z^2 + \cdots \in \cO_p[[z]]
$$
We find that $|b_{0}|_p \leq 1, |b_{1}|_p=1$,
and $|b_i|_p\leq p^{1-i}$ for each $i\geq 2$.
By considering the iterates of the map
$z\mapsto b_0 + b_1 z$, we have
$g^k(z) \equiv z \pmod{p}$ 
for some $1\leq k\leq q$.
Hence,
$g^k(z) = a_0 + a_1z + a_2 z^2 + \cdots$
satisfies the hypotheses of
Theorem~\ref{thm:analytic-congruent-to-z},
giving a power series $\tilde{u}\in\cO_p[[z]]$
mapping $\Dbar(0,1)$ into itself,
with $\tilde{u}(0)=0$
and $\tilde{u}(z+1) = g^k(\tilde{u}(z))$.
It follows that $u(z) = p\tilde{u}(z)$
has the desired properties.
\end{proof}

Now we are ready to prove Theorem~\ref{thm:controlling-c-n-for-curves}.

\begin{proof}[Proof of Theorem~\ref{thm:controlling-c-n-for-curves}]
For simplicity we assume that $X = \Pl^{1} \times \Pl^{1}$,
and that $V
\subset X$ is an irreducible curve; the argument is easily modified
to include the general case.
If $x_i$ is preperiodic under $f_i$ for either $i=1$ or $i=2$,
the result is trivial.
If both $f_{1}$ and $f_{2}$ are of degree one, $V$
can be shown to be
be periodic, either by the Skolem-Mahler-Lech theorem, or by
\cite[Theorem 3.4]{Par}.
Thus, possibly after permuting indices, we may assume
that the 
degree of $f_{1}$ is greater than $1$.
%
Define $\pi_1 : V \to \Pl^1(K)$ by
$(z_{1},z_{2}) \to z_{1}$.
By taking an periodic cycle $D=\{d_1,\ldots,d_a\}$ of $f_{1}$
of sufficiently large cardinality $a$,
defined over some number field $L$, we 
may assume that 
$D$ is not superattracting (i.e., no $d_i$ is a critical
point of $f_1$), that
all points $(\alpha_{1},\alpha_{2}) \in \pi_1^{-1}(D) \cap V$ are
smooth points on $V$, and finally,
that for $(z_{1},z_{2})$ near $(\alpha_1,\alpha_2)$,
\begin{equation}
  \label{eq:x1-x2-curve-linear}
z_{1}-\alpha_{1} = \gamma_{\alpha} \cdot (z_{2}-\alpha_{2}) + 
O\left( (z_{2}-\alpha_{2})^{2} \right),
\end{equation}
for some $\gamma_{\alpha}\ne 0$.
(Note that only finitely many points violate these conditions.)  
Since $f_{1}$ is not preperiodic, by \cite[Theorem~2.2]{SilSiegel} (or
\cite[Lemma~6.1]{Par}), we can find infinitely many
primes $p$ such 
that $|f_{1}^{n}(x_{1})-d_{1}|_{p}<1$ for some $n$,
where $|\cdot|_p$ denotes some extension of
the $p$-adic absolute value on $\Q$ to $L$.  We may
of course assume that $L/\Q$ is unramified at $p$
and that $|\gamma_\alpha|_p = |(f_1^a)'(d_1)|_p=1$
for all sufficiently large $p$, as there are only
finitely many $p$ not fitting these conditions.
In particular, the orbit of $x_{1}$ under $f_1$
ends up in a domain of quasiperiodicity.

If the orbit of $x_{2}$ under $f_2$ also has quasiperiodic behavior,
then $V$
is periodic by \cite[Theorem~3.4]{Par}.  Otherwise, the
orbit of $x_2$ ends up in an attracting or superattracting domain.
The arguments
in these two cases are very similar, and we shall only give details
for the attracting case.
Hence, assume that $f_{2}^{n}(x_{2})$ tends
to an attracting cycle $E = \{e_1, e_2, \ldots, e_{b}\}$, with
multiplier $\lambda_{2}$ satisfying $0<|\lambda_2|_{p}<1$.  Since
$\lambda_{2}$ and $E$ are defined over $K_{p}$, and $K_p/\Q_p$
is unramified, we have $|\lambda|_{p}\leq 1/p$.
Note that $b \leq p^{[K_{p}:\Q_{p}]}+1 \leq
p^{[K:\Q]}+1$.  Let $N = \lcm(a,b)$, so that
$N \leq a \cdot (p^{[K:\Q]}+1) = O(p^{[K:\Q]})$.
Choose representatives
$\{\alpha_{ij}:1\leq i \leq a, \, 1\leq j \leq b\}$
for $\Z/N\Z$ such that
$$
|f_{1}^{\alpha_{ij}+Nn}(x_{1})-d_{i}|_{p}<1, \quad
|f_{2}^{\alpha_{ij}+Nn}(x_{2})-e_{j}|_{p}<1
$$
for $n$ sufficiently large.  
At the cost of increasing $a$ (and hence $N$)
by a factor bounded by $p^{[K:\Q]}$, by
Proposition~\ref{thm:bounding-k} 
and Lemma~\ref{lem:attr}
there exist $p$-adic power series $A_{i},B_{j}$,
such that
\begin{equation}
\label{eq:f1-f2-formula}
f_{1}^{\alpha_{ij}+Nn}(x_{1})-d_{i} =
A_{i}(k), \quad
f_{2}^{\alpha_{ij}+Nn}(x_{2})-e_{j} = 
B_{j}(\lambda_{2}^{k})
\end{equation}
for $n$ sufficiently large.
If $n>m$ and $\phi^{mN+\alpha_{ij}} (P),
\phi^{nN+\alpha_{ij}}(P) \in V$,
then (\ref{eq:x1-x2-curve-linear}) and
(\ref{eq:f1-f2-formula})  yield that
$$
|A_{i}(n) - A_{i}(m) |_{p} = 
O(|\lambda_{2}|_{p}^{m}),
$$
since we had $|\gamma_{\alpha}|_p=1$ in \eqref{eq:x1-x2-curve-linear}.
Hence $n \equiv m \pmod{p^{m-O_{p}(1)}}$, where the $O_{p}(1)$
depends on the derivative of $A_{i}$.
Thus, if we take
$C<p$, we find that $n \geq m+C^{m}$  for $m$
sufficiently large. 
\end{proof}

\section{An analytic counterexample}
\label{analytic example}

It is natural to ask if an even more rapid growth condition than the
one in Theorem~\ref{thm:ml} should hold when $V$ is not periodic.
However, as 
the following shows, Lemma~\ref{lem:speeds}
is essentially sharp.
\begin{proposition}
\label{prop:badex}
For any prime $p\geq 2$ and for any positive integer $n_1$,
there is an increasing sequence $\{n_j\}_{j\geq 2}$
of positive integers and a power series $f(z) \in \Zp[[z]]$ such that
$$
f(p^{n_j}) = n_j
\quad\text{and}\quad
n_j + p^{n_j} \leq n_{j+1} \leq n_j + p^{n_1+\cdots+n_j}$$
for all $j\geq 1$.  Moreover,
$n_1+\cdots + n_{j-1} \leq n_j$,
and hence $n_{j+1} \leq n_j + p^{2n_j}$.
\end{proposition}

\begin{remark}
Setting $G(z_0,z_1)=z_1 - f(z_0)$, we find that Lemma~\ref{lem:speeds}
cannot be substantially improved; specifically, the constant $C$ is at
most $p^2$ for this example.
Furthermore, the bound of $p^2$ can be improved to
something much closer to $p$ because, by a simple inductive argument,
one can show that for every $j\ge 1$, we have
$$n_1+\cdots + n_j \le \frac{n_{j+1}}{n_1},$$
from which
$n_{j+1} \leq n_j + p^{n_j\cdot \left( 1+ \frac{1}{n_1}\right)}$
follows. Letting $n_1$ be arbitrarily large we obtain
that for every $\epsilon>0$ there exists an increasing sequence
$\{n_j\}_{j\ge 1}$ satisfying the hypothesis of
Proposition~\ref{prop:badex}, and for which
$$n_j + p^{n_j} \leq n_{j+1} \leq n_j + p^{(1+\epsilon)\cdot n_j}.$$
\end{remark}

\begin{proof}[Proof of Proposition~\ref{prop:badex}.]

We will inductively construct the sequence $\{n_j:j\geq 2\}$ of positive
integers and a sequence $\{f_j(z):j\geq 1\}$
of polynomials $f_j\in\Zp[z]$,
with $\deg(f_j)=j-1$.  The power series $f$ will be
$f=\lim_{j\to\infty}f_j$.

Let $f_1$ be the constant polynomial $n_1$.
Then, for each $j\geq 1$, suppose we are already given $f_1,\ldots,f_j$
and $n_1,\ldots,n_j$ such that
$f_k(p^{n_i})=n_i$
for each $i,k$ with $1\leq i\leq k\leq j$.
Choose $n_{j+1}$ to be
the unique integer such that
$$n_j + 1 \leq n_{j+1} \leq n_j + p^{n_1+\cdots + n_j}$$
and
\begin{equation}
\label{eq:njbound}
|n_{j+1} - f_j(0)|_p \leq |p|_p^{n_1 + \cdots + n_j} .
\end{equation}
Note that because 
$f_j\in\Zp[z]$ and $f_j(p^{n_j})=n_j$, we have
$$|f_j(0)-n_j|_p = |f_j(0) - f_j(p^{n_j})|_p \leq |p|_p^{n_j},$$
and therefore $|n_{j+1}-n_j|_p \leq |p|_p^{n_j}$, implying that
$n_{j+1}\geq n_j + p^{n_j}$ and
that
$n_{j+1}\geq n_1 + n_2 + \cdots + n_j$,
as claimed in the Proposition.

Define
$g_j(z):=(z-p^{n_1})(z-p^{n_2})\cdots(z-p^{n_j})$, and set
$$c_j := \frac{\dsps n_{j+1} - f_j(p^{n_{j+1}})}{\dsps g_j(p^{n_{j+1}})}
\in\Qp,$$
and
$$f_{j+1}(z) := f_j(z) + c_j g_j(z)\in\Qp[z].$$


Clearly, $f_{j+1}(p^{n_i})=n_i$ for all $i=1,\ldots,j+1$.
We will show that $c_j\in\Zp$, and hence $f_{j+1}\in\Zp[z]$, completing
the induction.
In fact, because
(for any fixed $m\geq 0$)
the size of the $z^m$-coefficient of $g_j$
goes to zero as $j\to\infty$, it follows that
$\lim_{j\to\infty}f_j$ converges to some power series $f\in\Zp[[z]]$.
Moreover, we get $f(p^{n_i})=n_i$ for all $i\geq 1$.

Thus, it suffices to show $|c_j|_p\leq 1$.
However, we have 
\begin{equation}
\label{eq:fjbound}
|f_j(0) - f_j(p^{n_{j+1}})|_p \leq |p|_p^{n_{j+1}},
\end{equation}
because $f_j\in\Zp[[z]]$.  Therefore,
\begin{align*}
|n_{j+1} - f_j(p^{n_{j+1}})|_p &\leq
\max\{ |n_{j+1}-f_j(0)|_p , |f_j(0) - f_j(p^{n_{j+1}})|_p \}
\\
&\leq 
\max\{ |p|_p^{n_1 + \cdots + n_j} , |p|_p^{n_{j+1}} \} \\
& = |p|_p^{n_1 + \cdots + n_j} = |g_j(p^{n_{j+1}})|_p,
\end{align*}
where the second inequality is by \eqref{eq:njbound} and
\eqref{eq:fjbound}.
It follows immediately that $|c_j|_p \leq 1$, as desired.

\end{proof}

\def\cprime{$'$} \def\cprime{$'$} \def\cprime{$'$} \def\cprime{$'$}
\providecommand{\bysame}{\leavevmode\hbox to3em{\hrulefill}\thinspace}
\providecommand{\MR}{\relax\ifhmode\unskip\space\fi MR }
\providecommand{\MRhref}[2]{%
  \href{http://www.ams.org/mathscinet-getitem?mr=#1}{#2}
}
\providecommand{\href}[2]{#2}

\end{document}